\documentclass[11pt]{article}
\usepackage{amsmath,amsthm,amssymb}
\usepackage{fancyhdr}
\usepackage[british]{babel}
\usepackage{geometry,mathtools}
\usepackage{enumitem}
\usepackage{algpseudocode}
\usepackage{dsfont}
\usepackage{centernot}
\usepackage{xstring}
\usepackage{colortbl}
\usepackage{setspace}
\usepackage{algorithm}
\usepackage{graphicx}
\usepackage[font={footnotesize}]{caption}
\usepackage[usenames,dvipsnames,table]{xcolor}
\usepackage{pstricks,tikz}
\usepackage[h]{esvect}
\usepackage[
		bookmarksopen=true,
		bookmarksopenlevel=1,
		colorlinks=true,
		linkcolor=darkblue,
        linktoc=page,
		citecolor=darkblue,
]{hyperref}
\usepackage{bbm}
\usepackage{hyperref}
\usepackage{cleveref}
\usepackage[square,sort,comma,numbers]{natbib}
\definecolor{darkblue}{rgb}{0,0,0.5}

\newdimen\margin
\def\textno#1&#2\par{
   \margin=\hsize
   \advance\margin by -4\parindent
          \setbox1=\hbox{\sl#1}
   \ifdim\wd1 < \margin
      $$\box1\eqno#2$$
   \else
      \bigbreak
      \hbox to \hsize{\indent$\vcenter{\advance\hsize by -3\parindent
      \it\noindent#1}\hfil#2$}
      \bigbreak
   \fi}

\newtheorem*{claim*}{Claim}
\newtheorem{theorem}[algorithm]{Theorem}
\newtheorem{prop}[algorithm]{Proposition}
\newtheorem{lemma}[algorithm]{Lemma}

\theoremstyle{definition}

\newtheorem{conj}[algorithm]{Conjecture}

\crefname{claim}{Claim}{Claims}
\Crefname{claim}{Claim}{Claims}

\newcounter{stepenv}
\newenvironment{stepenv}[1][]{\refstepcounter{stepenv}}{}

\newcounter{step}[stepenv]

\newcounter{substep}[step]
\renewcommand{\thesubstep}{\thestep.\arabic{substep}}

\newcounter{claim}[stepenv]

\newcommand{\cA}{\mathcal{A}}
\newcommand{\cB}{\mathcal{B}}

\newcommand{\cF}{\mathcal{F}}

\newcommand{\defn}{\emph}

\def\sm{\setminus}
\newcommand{\Set}[1]{\{#1\}}

\def\In{\subseteq}

\def\COMMENT#1{}
\def\TASK#1{}
\let\TASK=\footnote             
\let\COMMENT=\footnote          

\begin{document}

\title{\vspace{-0.9cm}The intersection spectrum of $3$-chromatic intersecting hypergraphs}

\author{Matija Buci\'c \thanks{Department of Mathematics, ETH Z\"urich, Switzerland. Email: \href{mailto:matija.bucic@math.ethz.ch} {\nolinkurl{matija.bucic@math.ethz.ch}}.}
\and
Stefan Glock \thanks{Institute for Theoretical Studies, ETH Z\"urich, Switzerland.
Email: \href{mailto:dr.stefan.glock@gmail.com}{\nolinkurl{dr.stefan.glock@gmail.com}}.
Research supported by Dr. Max R\"ossler, the Walter Haefner Foundation and the ETH Z\"urich Foundation.}
\and Benny Sudakov \thanks{Department of Mathematics, ETH Z\"urich, Switzerland. Email:
\href{mailto:benny.sudakov@gmail.com} {\nolinkurl{benny.sudakov@gmail.com}}.
Research supported in part by SNSF grant 200021\_196965.}
}

\date{}

\maketitle

\begin{abstract}
For a hypergraph $H$, define its intersection spectrum $I(H)$ as the set of all intersection sizes $|E\cap F|$ of distinct edges $E,F\in E(H)$.
In their seminal paper from 1973 which introduced the local lemma, Erd\H{o}s and Lov\'asz asked: how large must the intersection spectrum of a $k$-uniform $3$-chromatic intersecting hypergraph be? They showed that such a hypergraph must have at least three intersection sizes, and conjectured that the size of the intersection spectrum tends to infinity with~$k$. Despite the problem being reiterated several times over the years by Erd\H{o}s and other researchers, the lower bound of three intersection sizes has remarkably withstood any improvement until now. 
In this paper, we prove the Erd\H{o}s--Lov\'asz conjecture in a strong form by showing that there are at least $k^{1/2-o(1)}$ intersection sizes.
Our proof consists of a delicate interplay between Ramsey type arguments and a density increment approach.


\end{abstract}

\section{Introduction}

A family $\cF$ of sets is said to have \defn{property B} if there exists a set $X$ which properly intersects every set of the family, that is, $\emptyset\neq F\cap X\neq F$ for all $F\in \cF$. The term was coined in the 1930s by Miller~\cite{miller:37a,miller:37b} in honor of Felix Bernstein. In 1908, Bernstein~\cite{bernstein:1908} proved that for any transfinite cardinal number $\kappa$, any family $\cF$ of cardinality at most $\kappa$, whose sets have cardinality at least $\kappa$, has property~B. 
In the 60s, Erd\H{o}s and Hajnal~\cite{EH:61} revived the study of property~B, and initiated its investigation for finite set systems, or hypergraphs. A hypergraph $H$ consists of a \defn{vertex set} $V(H)$ and an \defn{edge set} $E(H)$, where every edge is a subset of the vertex set. As usual, $H$ is \defn{$k$-uniform} if every edge has size~$k$. A hypergraph is \defn{$r$-colorable} if its vertices can be colored with $r$ colors such that no edge is monochromatic.
Note that a hypergraph has property B if and only if it is $2$-colorable. 

The famous problem of Erd\H{o}s and Hajnal is to determine $m(k)$, the minimum number of edges in a $k$-uniform hypergraph which is not $2$-colorable. This can be viewed as an analogue of Bernstein's result for finite cardinals.
Clearly, one has $m(k)\le \binom{2k-1}{k}$, since the family of all $k$-subsets of a given set of size $2k-1$ does not have property~B. On the other hand, Erd\H{o}s~\cite{erdos:63} soon observed that $m(k)\ge 2^{k-1}$. Indeed, if a hypergraph has less than $2^{k-1}$ edges, then the expected number of monochromatic edges in a random $2$-coloring is less than $1$, hence a proper $2$-coloring exists. Thanks to the effort of many researchers~\cite{schmidt:64,AM:64,erdos:64c,beck:77,beck:78,RS:00,CK:15,pluhar:09,spencer:81,gebauer:13}, the best known bounds are now
\begin{align}
	 \Omega\left(\sqrt{k/\log k}\right)  \le {m(k)}/{ 2^{k}} \le O(k^2),  \label{prop B bounds}
\end{align}
proofs of which are now textbook examples of the probabilistic method \cite{alon-spencer}. Improving either of these bounds would be of immense interest.

The $2$-colorability problem for hypergraphs has inspired a great amount of research over the last half century, with many deep results proved and methods developed. One outstanding example is the Lov\'asz local lemma, originally employed by Erd\H{o}s and Lov\'asz~\cite{EL:75} to show that a $k$-uniform hypergraph is $2$-colorable if every edge intersects at most $2^{k-3}$ other edges.
In addition to the Lov\'asz local lemma, the seminal paper of Erd\H{o}s and Lov\'asz~\cite{EL:75} from 1973 left behind a whole legacy of problems and results on the $2$-colorability problem.
Some of their problems were solved relatively soon~\cite{beck:77,beck:78}, others took decades~\cite{FOT:96,kahn:92,kahn:94} or are still the subject of ongoing research.

At the heart of some particularly notorious problems are intersecting hypergraphs.
In an \defn{intersecting} hypergraph (Erd\H{o}s and Lov\'asz called them \emph{cliques}), any two edges intersect in at least one vertex.
The study of intersecting families is a rich topic in itself, which has brought forth many important results such as the Erd\H{o}s--Ko--Rado theorem. We refer the interested reader to~\cite{FT:16}.
With regards to colorability, the intersecting property imposes strong restrictions. 
For instance, it is easy to see that any intersecting hypergraph has chromatic number at most~$3$. Hence, the $3$-chromatic ones are exactly those which do not have property~B. On the other hand, every $3$-chromatic intersecting hypergraph is ``critical'' in the sense that deleting just one edge makes it $2$-colorable. These and other reasons (explained below) make the $2$-colorability problem for intersecting hypergraphs very interesting. It motivated Erd\H{o}s and Lov\'asz to initiate the study of 
$3$-chromatic intersecting hypergraphs, proving some fundamental results and raising tantalizing questions.

Analogously to $m(k)$, define $\tilde{m}(k)$ as the minimum number of edges in a $k$-uniform intersecting hypergraph which is not $2$-colorable. The problem of estimating $\tilde{m}(k)$ seems much harder. While for non-intersecting hypergraphs, we know at least that $\lim_{k\to \infty}\sqrt[k]{m(k)}=2$, no such result is in sight for $\tilde{m}(k)$. Clearly, the lower bound in~\eqref{prop B bounds} also holds for $\tilde{m}(k)$. However, the best known upper bound for $\tilde{m}(k)$ is exponentially worse. For any $k$ which is a power of $3$, an iterative construction based on the Fano plane yields a $k$-uniform $3$-chromatic intersecting hypergraph with $7^{\frac{k-1}{2}}$ edges (see~\cite{AM:64,EL:75}). Perhaps the main obstacle to improving this bound is that the probabilistic method does not seem applicable for intersecting hypergraphs.
Erd\H{o}s and Lov\'asz also asked for the minimum number of edges in a $k$-uniform intersecting hypergraph with cover number~$k$, which can be viewed as a relaxation of $\tilde{m}(k)$ since any $k$-uniform $3$-chromatic intersecting hypergraph has cover number~$k$. This problem was famously solved by Kahn~\cite{kahn:94}.

In addition to the size of $3$-chromatic intersecting hypergraphs, Erd\H{o}s and Lov\'asz also studied their ``intersection spectrum''. For a hypergraph $H$, define $I(H)$ as the set of all intersection sizes $|E\cap F|$ of distinct edges $E,F\in E(H)$.
A folklore observation is that if a hypergraph is not $2$-colorable, then there must be two edges which intersect in exactly one vertex, that is, $1\in I(H)$. A very natural question is what else we can say about the intersection spectrum of a non-$2$-colorable hypergraph.
In general, hypergraphs can be non-$2$-colorable even if their only intersection sizes are $0$ and~$1$. There are various basic constructions for this, see e.g.~\cite{lovasz:79}.
For instance, consider $K_N^{k-1}$, the complete $(k-1)$-uniform hypergraph on $N$ vertices. For $N$ large enough, any $2$-coloring of the edges will contain a monochromatic clique on $k$ vertices by Ramsey's theorem. Let $H$ be the hypergraph with $V(H)=E(K_N^{k-1})$ whose edges correspond to the $k$-cliques of $K_N^{k-1}$. Then $H$ is a $k$-uniform non-$2$-colorable hypergraph with $I(H)= \Set{0,1}$.

Erd\H{o}s and Lov\'asz observed that the situation changes drastically for intersecting hypergraphs. 
In the aforementioned construction of the iterated Fano plane, the intersection spectrum consists of all odd numbers (between $1$ and $k-1$). In particular, the \defn{maximal intersection size} is $k-2$, and the \defn{number of intersection sizes} is $(k-1)/2$. Astonishingly, not a single example (of a $k$-uniform $3$-chromatic intersecting hypergraph) is known where these quantities are any smaller.
Intrigued by this, Erd\H{o}s and Lov\'asz studied the corresponding lower bounds. Concerning the maximal intersection size, they (and also Shelah) could prove that $\max I(H)= \Omega(k/\log k)$ for any $k$-uniform $3$-chromatic intersecting hypergraph~$H$. This is in stark contrast to non-intersecting hypergraphs where we can have $\max I(H)=1$ as discussed. In fact, Erd\H{o}s and Lov\'asz conjectured that a linear bound should hold, or perhaps even $k-O(1)$. Erd\H{o}s~\cite{erdos:94} later offered \$100 for settling this question.

Finally, consider the number of intersection sizes. As already noted, we always have $1\in I(H)$. Moreover, the above result on $\max I(H)$ adds another intersection size for sufficiently large~$k$. Hence, $3$-chromatic intersecting hypergraphs have a small intersection size, namely~$1$, and a relatively big intersection size. Recall that general non-$2$-colorable hypergraphs might only have two intersection sizes. However, Erd\H{o}s and Lov\'asz were able to show that intersecting hypergraphs must have at least one more. Using a theorem of Deza~\cite{deza:74} on sunflowers, they proved that $i(k)\ge 3$ for sufficiently large~$k$, where $i(k)$ is the minimum of $|I(H)|$ over all $k$-uniform $3$-chromatic intersecting hypergraphs. They also remarked that they ``cannot even prove'' that $i(k)$ tends to infinity. This is particularly striking in view of the best known upper bound being $(k-1)/2$.

\begin{conj}[Erd\H{o}s and Lov\'asz, 1973] \label{conj:EL:75}
$i(k)\to \infty$ as $k\to \infty$.
\end{conj}

Despite the fact that over the years this problem has been reiterated many times by Erd\H{o}s and other researchers~\cite{erdos:79,erdos:94,erdos:94b,CG:98,RC:20},
remarkably, the lower bound of three intersection sizes has withstood any improvement until now.
In this paper, we prove Conjecture~\ref{conj:EL:75} in the following strong form.

\begin{theorem} \label{thm:main}
The intersection spectrum of a $k$-uniform $3$-chromatic intersecting hypergraph has size at least $\Omega(k^{1/2}/\log k)$.
\end{theorem}

\section{Intersection spectrum of $3$-chromatic intersecting hypergraphs}

In this section we will prove \Cref{thm:main}. We begin by gathering several properties of our hypergraphs that we are going to use. 

\subsection{Preliminaries}
The following is a classical result of Erd\H{o}s~\cite{erdos:63}, already mentioned in the introduction.

\begin{theorem}\label{thm:erdos prob}
	Any non $2$-colorable, $k$-uniform hypergraph has at least $2^{k-1}$ edges.
\end{theorem}

This result provides us with a large number of edges among which we can look for different intersection sizes. The following simple result (inspired by Problem 13.16 in \cite{lovasz:79}) will give rise to certain restrictions on the distribution of intersection sizes across our hypergraph, which in turn will enable us to use a density increment argument. 

\begin{prop}\label{prop:CS}
	Let $\cA$ be a $k$-uniform and $\cB$ a $k'$-uniform hypergraph on the same vertex set and with the same number of edges $\ell$. Then
	$$\sum_{\Set{A,A'}\In  E(\cA)} |A\cap A'| +  \sum_{\Set{B,B'}\In  E(\cB)} |B\cap B'| \ge \sum_{A\in E(\cA),B\in E(\cB)} |A\cap B| - \ell(k+k')/2.$$
\end{prop}

\begin{proof}
	For any vertex $x$ let us denote by $a_x$ and $b_x$ the number of edges in $\cA$ and $\cB$ that contain $x$, respectively. Each $x$ will be a common vertex of exactly $\binom{a_x}{2},\binom{b_x}{2}$ and $a_xb_x$ pairs of edges both in $\cA$, both in $\cB$ and one each in $\cA, \cB$, respectively. Observe further that $\sum_x a_x = \ell k$ and $\sum_x b_x = \ell k'$. Putting these observations together, the desired inequality is equivalent to 
	$$\sum_x \binom{a_x}{2}+ \sum_x \binom{b_x}{2} \ge \sum_x a_xb_x-\sum_x (a_x+b_x)/2.$$
	Since $a_x^2+b_x^2\ge 2a_xb_x$, this completes the proof.
\end{proof}

It will be convenient for us to introduce the following averaging functions.
Given a hypergraph $H$ and disjoint subsets $S,T \subseteq E(H)$ we define $$\lambda_S:= \frac{1}{\binom{|S|}{2}}\sum_{\{e,f\}\subseteq S} |e \cap f| \quad \mbox{and} \quad \lambda_{S,T}:= \frac{1}{|S||T|} \sum_{e \in S, f \in T} |e \cap f|.$$
The following lemma encapsulates the aforementioned restrictions on the distribution of intersection sizes across a $k$-uniform hypergraph.
Roughly speaking, it says that given two disjoint subsets of edges, on average, the intersection sizes inside the sets are at least as big as across. Moreover, a crucial ingredient of our density increment approach is that we can obtain a stronger inequality if we can find some vertices that exclusively belong to the edges of one set.

\begin{lemma}\label{lem:average-lambda}
Let $S,T$ be disjoint collections of $\ell$ edges of a $k$-uniform hypergraph $H$, with the property that there are $x$ vertices of $H$ which all belong to every edge in $S$ and none of them belong to any edge in~$T$. Then  $$\frac{\lambda_S+\lambda_T}2 \ge \lambda_{S,T}+\frac x2-\frac{k}{\ell-1}.$$
\end{lemma}

\begin{proof}
Let $\mathcal{A}$ be the $(k-x)$-uniform hypergraph with edge set consisting of the edges in $S$ with their $x$ common vertices disjoint from any edge in $T$ removed. Let $\mathcal{B}$ be the $k$-uniform hypergraph consisting of the edges in~$T$. Since the intersection size drops by $x$ for pairs of edges in $S$ and it stays the same for pairs in $T$ or one in $S$ and one in $T$, \Cref{prop:CS} gives
\begin{align*} 
\sum_{\{e,f\}\subseteq S} (|e \cap f|-x)+ \sum_{\{e,f\}\subseteq T} |e \cap f| &\ge \sum_{e\in S, f\in T} |e \cap f|  -\ell(2k-x)/2.
\end{align*}
Upon dividing by $\binom{\ell}{2}$ we obtain ${\lambda_S+\lambda_T}  \ge x+\frac{2\ell}{\ell-1}\lambda_{S,T}-\frac{2k-x}{\ell-1}\ge x+2\lambda_{S,T}-\frac{2k}{\ell-1}$.
\end{proof}

The following useful fact, which exploits the property of being 3-chromatic and intersecting, imposes further restrictions on the distribution of intersection sizes. In particular, it allows us to find many edges with large common intersection. 

\begin{prop}\label{prop:greedy increase}
Let $H$ be a $k$-uniform $3$-chromatic and intersecting hypergraph, and $X\In V(H)$. Then for any $0\le i\le k-|X|$, there exists a set $X_i\In V(H)$ of size $|X|+i$ such that at least a $k^{-i}$ proportion of the edges containing $X$ also contain~$X_i$.
\end{prop}

\begin{proof}
We proceed by induction on~$i$. The case $i=0$ is trivial. Suppose now we have $X_i$ with $|X_i|<k$. Since $H$ is $3$-chromatic it must contain an edge $Y$ disjoint from $X_i$, as otherwise we can color vertices in $X_i$ red and all vertices outside blue to obtain a proper $2$-coloring. Since $H$ is intersecting every edge containing $X_i$ intersects~$Y$. By averaging, some vertex $x\in Y$ is contained in at least a $\frac{1}{k}$ proportion of these edges. Thus, adding $x$ to $X_i$ yields the desired $X_{i+1}$.
\end{proof}
 
\noindent \textbf{Remark.} The above proposition does not require the full strength of the $3$-chromatic and intersecting assumptions. In fact we can replace them with the assumption that for any set $X\In V(H)$ of size less than~$k$, there exists a set $Y\In V(H)\sm X$ of size $k$ such that every edge of $H$ intersects~$Y$. Since this proposition and \Cref{thm:erdos prob} are the only places where we will use these assumptions in our proof of \Cref{thm:main}, we note that it actually holds if we only assume that this alternative property holds and that $H$ has at least $2^{k-1}$ edges.

In our proof of \Cref{thm:main}, we will make use of a dependent random choice lemma.
Dependent random choice is a powerful probabilistic technique which has recently led to a number of advances in Ramsey theory, extremal graph theory, additive combinatorics, and combinatorial geometry. Early variants of this technique were developed by Gowers~\cite{gowers:98}, Kostochka and R\"odl~\cite{KR:01} and Sudakov~\cite{sudakov:03}. In many applications, the technique is used to prove the useful fact that every dense graph contains a large subset $U$ in which almost every set of $t$ vertices has many common neighbors. For more information about dependent random choice and its applications, we refer the interested reader to the survey~\cite{FS:11}. We are going to use the following variant of the dependent random choice lemma (Lemma~6.3 in~\cite{FS:11}).

\begin{lemma}\label{lem:dependent}
If $d>0$, $t \leq n$ are positive integers, and $G$ is a graph with $m > 4td^{-t}n$ vertices and at least $d m^2/2$ edges, then there is a vertex subset $U$ with $|U| > 2n$ such that the proportion of $t$-subsets of $U$ with less than $n$ common neighbors in $G$ is less than $(2t)^{-t}$.
\end{lemma}

\subsection{Proof ideas}

In this subsection we will illustrate our proof ideas by sketching a slightly simpler argument which shows $|I(H)| \ge k^{1/3-o(1)}$. Our proof of \Cref{thm:main}, which will be presented in the next subsection, follows along very similar lines, with the exception of using dependent random choice in place of certain Ramsey arguments which we use here.

Let $H$ be a $k$-uniform, $3$-chromatic and intersecting hypergraph. Let $ \lambda_1 < \ldots < \lambda_r$ denote the distinct intersection sizes in $H$, so $r=|I(H)|$. A natural way to approach our problem is to define a coloring of the complete graph with vertex set $E(H)$ where an edge is colored according to the size of the intersection of its endpoints. We will refer to this coloring as the \textit{intersection coloring}. \Cref{thm:erdos prob} tells us that this graph is quite big, so if the number of colors $r$ is small we could hope to use Ramsey's theorem to find a number of edges which make a monochromatic clique in the intersection coloring, i.e.\ all pairwise intersection sizes are the same. This would lead to a contradiction if there were about $k^2$ such edges (this is an easy consequence of \cite{deza:74}). Unfortunately, in an arbitrary coloring we cannot find this many, even if we assume $|I(H)|=2$. 

However, the well-known argument for bounding Ramsey numbers actually gives us more than just a monochromatic clique. If we repeatedly take out an arbitrary edge of $H$ and only keep its majority color neighbors, we keep at least a proportion of $1/r$ of the edges per iteration. If we repeat $rt$ many times we can find a set $X$ consisting of $t$ edges that we took out which had the same majority color, so in particular $X$ is a monochromatic clique in the intersection coloring. Furthermore, we know that the size of the set of remaining edges $Y$ has lost at most a factor of $r^{rt}$ compared to the original number of edges. In addition, the complete bipartite graph between $X$ and $Y$ is also monochromatic in the same color as $X$.

In particular, this provides us with a pair $(X,Y)$ of disjoint subsets of edges of $H$ with the property that any two edges in $X$ as well as any pair of edges one in $X$ and one in $Y$ intersect in exactly $\lambda_i$ vertices, for some $\lambda_i$. We call such a pair a $\lambda_i$-pair. Note that if we choose $|X|=t \approx k^{1/3}$ and assume $r\le O(k^{1/3}/\log k)$ (otherwise we are done) then $|Y| \ge |E(H)|/r^{rt} \ge |E(H)|/k^{O(k^{2/3})}.$ 

Our strategy will be to show that given a $\lambda_i$-pair one can find a $\lambda_j$-pair, for some $j>i,$ whose set $X$ still has size $t$ and the size of $Y$ shrinks by at most a factor of $k^{O(k^{2/3})}$. Since by \Cref{thm:erdos prob} we know $|E(H)| \ge 2^{k-1}$, we can repeat this procedure at least $\Omega(k^{1/3}/\log k)$ times to conclude there are at least this many different intersection sizes and complete the proof.

To do this, let $(X,Y)$ be a $\lambda_i$-pair with $|X|=t\approx k^{1/3}$. We can apply \Cref{lem:average-lambda} to $X$ and any $t$-subset $Y'\subseteq Y$ to conclude that $\lambda_{Y'} \ge \lambda_i-2k^{2/3}$. So in particular, the average intersection size in any subset of $Y$ of size at least $t$ cannot be much lower than~$\lambda_i$. Next we take an arbitrary edge $U$ in $X$ and consider the intersections of edges in $Y$ with $U$. We will separate between two cases depending on the structure of $Y$. 

In the first case, many of the edges in $Y$ have almost the same intersection with~$U$. In this case we will find a collection of at least $|Y|/k^{O(k^{2/3})}$ edges in $Y$ which contain the same set of vertices of size at least $\lambda_i-x$, where $x\approx 10k^{2/3}$. Then we apply \Cref{prop:greedy increase} (with $i=x+1$) to obtain a subset of edges of size at least $|Y|/k^{O(k^{2/3})}$ in which any pair of edges intersects in more than $\lambda_i$ vertices. Applying once again the Ramsey argument, this time within this collection of edges, we find a $\lambda_j$-pair in which we only lost another factor of $k^{O(k^{2/3})}$ in terms of size of $Y$. Since all intersection sizes are larger than $\lambda_i$ we know that $j>i$, so we found our desired new pair.  

In the second case, the intersections of edges in $Y$ with $U$ are ``spread out''. Then we can find two disjoint subsets $A,B\In Y$ both of size $|Y|/k^{O(k^{2/3})}$ with the property that there is a set of $x$ vertices $W\subseteq U$ which belongs to every edge of $A$ and is disjoint from all edges in~$B$. By applying the Ramsey argument to the collection $A$ and to the collection $B$ we either find a desired $\lambda_j$-pair with $j>i$ or we find a $t$-subset $S\In A$ and a $t$-subset $T\In B$ such that all pairwise intersections inside $S$ and $T$ have size at most~$\lambda_i$. In particular, $\lambda_S,\lambda_T \le \lambda_i$. We now apply \Cref{lem:average-lambda} to $S$ and $T$, knowing that the $x=10k^{2/3}$ vertices in $W$ belong to every edge in $S$ and none belong to any edge in $T$. This will give us $\lambda_{S,T} \le \lambda_i-4k^{2/3}$. Combining these three inequalities we obtain $\lambda_{S \cup T} < \lambda_i-2k^{2/3}$, which contradicts our lower bound on the average intersection size among subsets of $Y$ and completes the argument.

\noindent \textbf{Remark.} While one can develop the Ramsey type arguments of this section to prove \Cref{thm:main}, in the following subsection we choose to present the argument based on dependent random choice as it demonstrates a slightly different approach, is slightly shorter and we believe has greater potential for further improvement.

\subsection{Proof of Theorem~\ref{thm:main}}

Let $H$ be a $k$-uniform, $3$-chromatic and intersecting hypergraph where we assume throughout the proof that $k$ is sufficiently large. 
Let $ \lambda_1<\ldots < \lambda_r$ denote the distinct intersection sizes in $H$, and set $\lambda_{r+1}=k$. Let us assume for the sake of contradiction that $r <\frac{\sqrt{k}}{51\log_2 k}.$  
Let us set $t=2\lceil \sqrt{k} \rceil$ and for all $1 \le i \le r+1$ 
\begin{equation}\label{eqn:m-i}
    m_i:=\frac{|E(H)|}{k^{25(i-1)t}}\ge 2^{k-1-25(i-1)t\log_2 k} \ge t,
\end{equation} where we used \Cref{thm:erdos prob} in the first inequality and $i-1 \le r < \frac{\sqrt{k}}{51\log_2 k}$ in the second. Our strategy is as follows. We will choose the largest $i$ such that we can find a subset $A\subseteq E(H)$ of size $m_i$ with the property that many pairs of edges in $A$ intersect in at least $\lambda_i$ vertices. We will then find such a subset for a larger $i$, reaching a contradiction. 

Let us first specify what we mean by many pairs above. In order to make use of the dependent random choice lemma we will quantify it in terms of how many $t$-subsets of $A$ consist of edges with all their pairwise intersections being of size smaller than $\lambda_{i}$. We call a set of edges \emph{$\lambda_i$-small} if all their pairwise intersections have size strictly smaller than $\lambda_i$. Let $i$ be the largest index such that there exists a subset $A\subseteq E(H)$ of size at least $m_i$ with the property that at most half of the $t$-subsets of $A$ are $\lambda_i$-small. Observe that since all intersections of edges in $H$ are of size at least $\lambda_1$ there is no set of $t$ edges of $H$ which is $\lambda_1$-small. Hence, by taking $A=E(H)$ we get that $i=1$ satisfies our condition, showing that $i$ exists. 

Our first goal is to show the following claim.

\begin{claim*}
There exists a pair $(X,Y)$ of disjoint subsets of $E(H)$ such that
\begin{enumerate}
    \item $|X|=t$ and $|Y|\ge {m_i}/{k^{3t}},$ 
    \item any two edges in $X$ intersect in at most $\lambda_i$ vertices, 
    \item any edge in $X$ and any edge in $Y$ intersect in at least $\lambda_i$ vertices.  
\end{enumerate}
\end{claim*} 

\begin{proof}
Let $m:=|A| \ge m_i$. By our choice of $A$, there are at least $ \frac{1}{2}\binom{m}{t}$ $t$-subsets of $A$ which contain at least one pair of edges with the intersection size at least $\lambda_i$. Since every pair of edges belongs to at most $\binom{m-2}{t-2}$ $t$-subsets of $A$, there are at least $\frac{1}{2}\binom{m}{t}/\binom{m-2}{t-2}=\frac{1}{t(t-1)} \binom {m}{2}\ge \frac{1}{8k} \frac{m^2}{2}$ pairs of edges in $A$ with intersections of size at least $\lambda_i$. 

We apply the dependent random choice lemma (\Cref{lem:dependent}) to the graph with vertex set $A$ and edge set consisting of pairs of elements in $A$ (so edges of $H$) which intersect in at least $\lambda_i$ vertices. Choosing $d=\frac{1}{8k}$ we may take 
$$n=\frac{m}{5td^{-t}}\ge \frac{m_i}{5t\cdot (8k)^{t}} \ge \frac{m_i}{k^{3t}} \ge m_{i+1} \ge t,$$ 
where in the last two inequalities we used \eqref{eqn:m-i}.

This provides us with a subset $A'\subseteq A$ of size at least $n \ge m_{i+1}$ such that all but an $(2t)^{-t}$ proportion of $t$-subsets of $A'$ have the property that there exist $n$ edges of $H$ each of which intersects every edge in the $t$-subset in at least $\lambda_i$ vertices. Since $|A'| \ge m_{i+1}$, by our maximality assumption on $i$ we know that more than half of the $t$-subsets of $A'$ must be $\lambda_{i+1}$-small, i.e.\ have all pairs of edges intersecting in less than $\lambda_{i+1}$, so at most $\lambda_i$ vertices.  Therefore, there exists a $t$-subset $X \subseteq A'$ which is both $\lambda_{i+1}$-small and there is a set $Y$ consisting of $n$ edges of $H$ with the property that any edge in $X$ intersects any edge in $Y$ in at least $\lambda_i$ vertices. These $X$ and $Y$ satisfy the desired properties.
\end{proof}

Our next step is to analyze average intersections among subsets of $Y$. The following claim tells us that any $t$-subset of $Y$ must have average intersection size almost as big as $\lambda_i$.

\begin{claim*}
For any $Y'\subseteq Y$ of size $t$ we have $\lambda_{Y'} \ge \lambda_i-2\sqrt{k}$.
\end{claim*}
\begin{proof}
\Cref{lem:average-lambda} applied to $X$ and $Y'$ with $x=0$ gives
$$\lambda_{Y'} \ge 2\lambda_{X,Y'}-\lambda_X-\frac{2k}{t-1} \ge \lambda_i-2\sqrt{k},$$
since $\lambda_X \le \lambda_i$, $\lambda_{X,Y'}\ge \lambda_i$ and by definition of~$t$. 
\end{proof}

On the other hand our maximality assumption on $i$ and the fact $Y$ is large allows us to find a subset $Y'$ as in the claim above which is $\lambda_{i+1}$-small, i.e.\ with $\lambda_{Y'}\le \lambda_i$. This does not immediately lead to a contradiction since we lost a little bit in the application of \Cref{lem:average-lambda} above.
Our final step is to show one can find such $Y'$ with an even smaller average intersection size which will give us a contradiction. 

To this end let us fix an edge $U \in X$ and let $(A_1,B_1,X_1),\dots, (A_m,B_m,X_m) $ be a collection of triples with the following properties:
\begin{enumerate}  [label=\textbullet] 
\item $A_1,\dots,A_m,B_1,\dots,B_m$ are distinct edges of $H$.
\item For each $i \in [m]$, $X_i$ is a subset of $U$ of size $x=10t$, and $X_i\In A_i$ but $X_i\cap B_i=\emptyset$.
\item $m$ is maximal subject to the above conditions.
\end{enumerate}

If $m<|Y|/4$ there is a subset $Y' \subseteq Y$ consisting of at least $|Y|/2$ edges with the property that no triple $(A_i,B_i,X_i)$ as above exists inside $Y'$. Let us fix an $A \in Y'$. If we let $A'=A \cap U$ we know $|A'|\ge \lambda_i$. For any edge $B \in Y'$ we know that $|A' \setminus B|<x$, since otherwise we could have extended our family by using $A,B$ and any $x$ vertices in $A' \setminus B$. This condition can be rewritten as $|A' \cap B| > |A'|-x.$ In particular, assuming $\lambda_i \ge x$, there exists a subset $A'' \subseteq A'$ of size $|A'|-x \ge \lambda_i-x$, which belongs to at least
$$\frac{|Y'|}{\binom{|A'|}{|A'|-x}} \ge \frac{|Y|/2}{\binom{k}{x}} \ge \frac{m_i}{k^{x+3t}}=\frac{m_i}{k^{13t}}$$
edges of $H$. If $\lambda_i<x$ we take $A''=\emptyset$. Either way, using \Cref{prop:greedy increase} (with $X=A''$ and $i=x+1$) we obtain a set of vertices of size larger than $\lambda_i$ contained in at least $m_i/k^{x+1+13t}\ge m_i/k^{25t} \ge m_{i+1}$ edges of $H$. Since this means that any two edges in this collection intersect in more than $\lambda_i$ vertices, we conclude that no set of $t$ edges from this collection is $\lambda_{i+1}$-small. This contradicts the maximality of~$i$.

Let us now assume $m \ge |Y|/4$. Since there are $\binom{|U|}{x}$ different choices for $X_i$, there are at least ${m}/{\binom{|U|}{x}} \ge \frac{|Y|}{4}/{\binom{k}{x}}\ge |Y|/k^{x}\ge m_{i+1}$ triples with the same $X_i$, which we denote by $X'$.
This and the maximality of $i$ allow us to find a set $S$ of size $t$, consisting of edges $A_i$ for which $X_i=X'$, which is $\lambda_{i+1}$-small, i.e.\ all pairs intersect in at most $\lambda_{i}$ vertices. Similarly we find a set $T$ of size $t$, consisting of edges $B_i$ for which $X_i=X'$, which is $\lambda_{i+1}$-small. Let us now remove half of the edges from both $S$ and $T$ so that $|S \cup T|=t$. The above claim then tells us that $\lambda_{S \cup T} \ge \lambda_i-2\sqrt{k}$. On the other hand \Cref{lem:average-lambda} applied to $S$ and $T$ with $x=|X'|$ gives $$ \frac{\lambda_S+\lambda_T}{2} \ge \lambda_{S,T}+\frac x2-\frac{k}{t/2-1}\ge  \lambda_{S,T}+8\sqrt{k}.$$ Since $\lambda_S,\lambda_T \le \lambda_i$ we obtain $\lambda_{S,T} \le \lambda_i-8\sqrt{k}$. Combining this with $\lambda_S,\lambda_T \le \lambda_i$ and $$\binom{t/2}{2}\lambda_S+\binom{t/2}{2}\lambda_T+(t/2)^2\lambda_{S,T}=\binom{t}{2}\lambda_{S \cup T}$$ we get $\lambda_{S \cup T} \le \lambda_i - \frac{(t/2)^2}{\binom{t}{2}} \cdot 8\sqrt{k}< \lambda_i-2\sqrt{k}$, which is a contradiction and completes the proof. \qed

\section{Concluding remarks}

We have proved the conjecture of Erd\H{o}s and Lov\'asz in a strong form, by showing that the intersection spectrum of any $k$-uniform $3$-chromatic intersecting hypergraph has size polynomial in $k$. It would be very interesting to improve our result and obtain a linear lower bound on the number of intersection sizes. In particular, this would also improve the result of Erd\H{o}s, Lov\'asz and Shelah on the maximal intersection size. Some of the ideas behind our arguments could be useful in obtaining a better understanding of how the intersection spectrum can look like in general as well.

\providecommand{\bysame}{\leavevmode\hbox to3em{\hrulefill}\thinspace}
\providecommand{\MR}{\relax\ifhmode\unskip\space\fi MR }
\providecommand{\MRhref}[2]{%
  \href{http://www.ams.org/mathscinet-getitem?mr=#1}{#2}
}
\providecommand{\href}[2]{#2}

\end{document}